\pdfoutput=1
\documentclass[graybox]{svmult}
\usepackage{amsthm, amsmath, amssymb, amsfonts,float, url, booktabs, tikz, setspace, fancyhdr, enumerate}
\usepackage{array} 
\usepackage{ booktabs, array, multirow, xcolor}
\usepackage{multirow}
\usepackage{makecell}
\usepackage{enumitem}
\usepackage[margin = 1in]{geometry}
\usepackage{pgfplots}
\pgfplotsset{compat=1.18}
\usepackage[czech,english]{babel} 
\usepackage[hidelinks]{hyperref}
\usepackage[capitalise]{cleveref}
\usepackage{tikz}
\usepackage{comment}
\usepackage{xcolor}
\usepackage{cite}
\usepackage[textsize=scriptsize,colorinlistoftodos]{todonotes}

\usepackage{makeidx}
\usepackage{graphicx}
\usepackage{multicol}
\usepackage[bottom]{footmisc}
\usepackage{newtxtext}
\usepackage[varvw]{newtxmath}

\theoremstyle{definition}
\newtheorem{construction}[theorem]{Construction}

\newcolumntype{C}[1]{>{\centering\arraybackslash}m{#1}}

\newcommand\restr[2]{{
  \left.\kern-\nulldelimiterspace 
  #1 
  \right|_{#2} 
  }}

\let\leq\leqslant
\let\geq\geqslant
\newcommand{\VC}{\dim_{\operatorname{VC}}}
\newcommand{\Tr}{\operatorname{Tr}}
\newcommand{\M}{m}

\makeindex

\begin{document}

\title*{Recent advances in arrow relations and traces of sets}

\author{Mingze Li, \\Jie Ma, and \\Mingyuan Rong}

\institute{Mingze Li \at School of Mathematical Sciences, University of Science and Technology of China, Hefei, Anhui 230026, China, \email{lmz10@mail.ustc.edu.cn}
\and Jie Ma \at School of Mathematical Sciences, University of Science and Technology of China, Hefei, Anhui 230026, China, and Yau Mathematical Sciences Center, Tsinghua University, Beijing 100084, China, \email{jiema@ustc.edu.cn}
\and Mingyuan Rong \at School of Mathematical Sciences, University of Science and Technology of China, Hefei, Anhui 230026, China,  \email{rong\_ming\_yuan@mail.ustc.edu.cn}}
\maketitle
\abstract{The arrow relation, a central concept in extremal set theory, captures quantitative relationships between families of sets and their traces. Formally, the arrow relation $(n, m) \rightarrow (a, b)$ signifies that for any family $\mathcal{F} \subseteq 2^{[n]}$ with $|\mathcal{F}| \geqslant m$, there exists an $a$-element subset $T \subseteq [n]$ such that the trace $\mathcal{F}_{|T} = \{ F \cap T : F \in \mathcal{F} \}$ contains at least $b$ distinct sets. 
This survey highlights recent progress on a variety of problems and results connected to arrow relations. 
We explore diverse topics, broadly categorized by different extremal perspectives on these relations, offering a cohesive overview of the field.}

\section{Introduction}\label{sec 1}

For a positive integer $n$, we let $[n]=\{1,2,\dots ,n\}$. We denote the power set of $[n]$ as $2^{[n]}=\{T: T\subseteq [n]\}$.
The study of finite sets and their extremal properties is a fundamental area in combinatorics. A key concept in this field is the \emph{trace} of a family of sets. For a family $\mathcal{F} \subseteq 2^{[n]}$, the trace of $\mathcal{F}$ on a subset $T \subseteq [n]$ is defined as $\mathcal{F}_{\mid T} = \{ F \cap T : F \in \mathcal{F} \}$. We say $T$ is \emph{shattered} by $\mathcal{F}$ if $\mathcal{F}_{\mid T} = 2^T$. A natural problem arises: How large must a family $\mathcal{F} \subseteq 2^{[n]}$ be to guarantee that it shatters some $k$-element subset?

This question was resolved in the early 1970s through independent work by Sauer~\cite{S1972}, Perles and Shelah~\cite{SP1972}, and Vapnik and Chervonenkis~\cite{V1968,VC2015}, confirming a conjecture of Erd\H{o}s~\cite{FP1991}.
Their result, now known as the Sauer-Shelah Lemma, gives a tight bound on the size of $\mathcal{F}$.

\begin{theorem}[Sauer-Shelah Lemma]\label{Sauer}
Let $n \geqslant k \geqslant 1$. For any family $\mathcal{F} \subseteq 2^{[n]}$ satisfying
$$
|\mathcal{F}| \geqslant \sum_{i=0}^{k-1} \binom{n}{i}+1,
$$
there exists a $k$-element set $T \subseteq [n]$ with $\mathcal{F}|_T = 2^T$. Moreover, this bound is tight.
\end{theorem}

In other words, this establishes that families of size at least $1 + \sum_{i=0}^{k-1} \binom{n}{i}$ must contain a trace of size $2^k$ on some $k$-element set. To study the general problem on traces of any size,  Hajnal~\cite{B1972} introduced the arrow relation through the following definition:

\begin{definition}
    For positive integers $n, m, a, b$, the \emph{arrow relation} $(n, m) \rightarrow (a, b)$ denotes that for every family $\mathcal{F} \subseteq 2^{[n]}$ with $|\mathcal{F}| \geqslant m$, there exists some $a$-element set $T \subseteq [n]$ such that $|\mathcal{F}_{\mid T}| \geqslant b$.
\end{definition}
Note that the Sauer-Shelah Lemma can be expressed via the following arrow relation
$$\left(n,\, 1+\sum_{i=0}^{k-1}\binom{n}{i}\right) \rightarrow (k,\, 2^{k}).$$ For convenience, for a family $\mathcal{F}\subseteq 2^{[n]}$, we also write $\mathcal{F} \rightarrow (a, b)$ if there exists an $a$-element subset $T \subseteq [n]$ with $|\mathcal{F}_{|T}| \geqslant b$. Conversely, $\mathcal{F} \not\rightarrow (a, b)$ signifies that for every $a$-element subset $T$, we have $|\mathcal{F}_{|T}| < b$. With these notions, $(n, m) \rightarrow (a, b)$ is equivalent to that $\mathcal{F} \rightarrow (a, b)$ holds for every family $\mathcal{F}\subseteq 2^{[n]}$ with $|\mathcal{F}|\geqslant m$.

Since the introduction of arrow relations, this framework has spurred a wide array of problems and results in extremal combinatorics. 
In recent years, research has primarily concentrated on exploring various extremal aspects of arrow relations.
A natural direction is to determine the minimum size of families to ensure a trace of given size, that is to determine the minimum $m$ such that $(n, m) \rightarrow (a, b)$ holds when $n$, $a$ and $b$ are given. Since the
case $b=2^{a}$ is solved by the Sauer-Shelah Lemma, it remains to consider $b<2^{a}$. The first step was made by Frankl~\cite{F1983}, who determined this bound for the $3$-vertex trace of size $7$. Recently, Frankl and Wang~\cite{FW2024} investigated the cases of $4$-vertex traces. 
Researchers are also interested in determining the maximum $b$ such that $(n, m) \rightarrow (a, b)$ holds when $n$, $m$ and $a$ are given.
This line of inquiry led to the introduction of the so-called \emph{trace function}, which was proposed and studied by Bollob\'{a}s and Radcliffe~\cite{BR1995}. Further advancements were achieved by Alon, Moshkovitz, and Solomon~\cite{AMS2019}.
Apart from these, another important direction is to consider the maximum $m$ such that $(n, m) \rightarrow (n-1, m-s)$ holds when $n$ and $s$ are given. Bondy \cite{B1972} and Bollob\'as \cite{L2007} first studied this problem, and later Frankl~\cite{F1983} presented a general method, resolving the cases where $s$ is just below powers of 2. Very recently, several new results in this area were presented in \cite{PS2021, LMR2024, RS2025}. 
There are several other topics closely related to arrow relations, which will be introduced in Sections \ref{6 Extremal families}, \ref{sec:forbidden_config}, and \ref{sec 8}.

This survey aims to provide an accessible entry point for researchers new to the field, serving as a useful resource for both orientation and reference. 
While we have made every effort to include representative works, the rapidly evolving nature of this area may have led to the omission of some valuable contributions.

We structure the discussion by grouping results based on their shared research themes. 
The paper is organized into the following sections: \ref{sec 1}. Introduction, \ref{sec 2}. A fundamental lemma of Frankl, \ref{sec 3}. Defect Sauer results (a term introduced in \cite{BR1995}), \ref{sec 4}. Single-element removal, \ref{sec 5}. Trace functions, \ref{6 Extremal families}. Extremal families, \ref{sec:forbidden_config}. Forbidden configurations in matrices, \ref{sec 8}. Tur\'{a}n numbers for traces, and \ref{sec 9}. Open problems.

\subsection*{Notation} Throughout this survey, we use the following notations:
\begin{itemize}
    \item Let $\mathbb{N}=\{0,1,2,\dots\}$ denote the set of all non-negative integers, $\mathbb{N}^{+}=\{1,2,3,\dots\}$ denote the set of all positive integers and $\mathbb{R}$ denote the set of all real numbers.
    \item For $n\in \mathbb{N}^{+}$, let $[n]=\{1,2,\dots,n\}$ denote the set of the first $n$ positive integers. For $a,b\in \mathbb{N}^{+}$ with $a\leq b$, let $[a:b]=\{a,a+1,\dots,b\}$ denote the set of consecutive integers from $a$ to $b$.
   \item For $s\in \mathbb{N}^{+}$ and a set $T$, let $2^T$ denote the collection of all subsets of $T$. Moreover, $\binom{T}{s}$ denotes the collection of all subsets of $T$ with exactly $s$ elements, and $\binom{T}{\leq s}$ denotes the collection of all subsets of $T$ with at most $s$ elements.
    \item For two sets $A$ and $B$, we write $A\triangle B=(A\setminus B)\cup (B\setminus A)$ as the \emph{symmetric difference} between $A$ and $B$.
    \item For a family $\mathcal{F}\subseteq 2^{[n]}$ and $v\in [n]$, let $d_{\mathcal{F}}(v)=|\{F\in \mathcal{F}: v\in F\}|$ denote the \emph{degree} of $v$ in $\mathcal{F}$ and $\delta(\mathcal{F})=\min_{w\in [n]}d_{\mathcal{F}}(w)$ denote the \emph{minimum degree} of $\mathcal{F}$.
    \item For a graph $G$, $\mbox{ex}(n, G)$ denotes the \emph{Tur\'{a}n number} of $G$, which is the maximum number of edges in an $n$-vertex graph that does not contain $G$ as a subgraph. For integers $n\geq k\geq 2$, the \emph{Tur\'an graph} $T(n,k)$ denotes the unique balanced complete $k$-partite graph on $n$ vertices.
    \item Similarly, for a $k$-uniform hypergraph $\mathcal{F}$, $\operatorname{ex}_k(n, \mathcal{F})$ denotes the maximum number of edges in an $n$-vertex $k$-uniform hypergraph which does not contain $\mathcal{F}$ as a subhypergraph.
    \item  The {\it Tur\'an density} of a $k$-uniform hypergraph  $\mathcal{F}$ is the limit $\pi(\mathcal{F})=\lim _{n \rightarrow \infty}\binom{n}{k}^{-1} \operatorname{ex}_k(n, \mathcal{F})$.
    \item For integers $n$, $m$, $a$ and $b$, the notion $(n, m) \not\rightarrow (a, b)$ denotes the negation of $(n, m) \rightarrow (a, b)$.
    \item Let $\mathbf{0}_{q}$ denote the $q \times 1$ matrix (i.e., a column vector of length $q$) whose entries are all $0$, $\mathbf{1}_{p}$ denote the $p \times 1$ matrix with all entries equal to $1$, and $\mathbf{I}_{k}$ denote the $k \times k$ identity matrix.
\end{itemize}

\section{A fundamental lemma of Frankl}\label{sec 2}
We begin by introducing a basic result of Frankl~\cite{F1983} on the study of arrow relation through the lens of a special but fundamental class of families: \emph{hereditary} families, which are closed under taking subsets. Formally, a family $\mathcal{F}$ is called \emph{hereditary} if for every $F' \subseteq F \in \mathcal{F}$, it holds that $F' \in \mathcal{F}$. Frankl~\cite{F1983} showed that to establish an arrow relation, it suffices to verify it for hereditary families, which are often more convenient to handle. 
This is formalized in the following lemma.

\begin{lemma}[Frankl~\cite{F1983}]\label{1.1}
The following statements are equivalent for any integers $n,m,a,b\in \mathbb{N}^{+}$:
\begin{itemize}
    \item $(n, m) \rightarrow (a, b)$.
    \item For every hereditary family $\mathcal{F} \subseteq 2^{[n]}$ with $|\mathcal{F}| = m$, $\mathcal{F} \rightarrow (a, b)$ holds.
\end{itemize}
\end{lemma}

In the rest of this section, we will present a proof of this lemma due to Frankl~\cite{F1983}. We first need to introduce an operation on families called \emph{squashing}. Frankl~\cite{F1983} and Alon~\cite{A1983} used squashing techniques to generalize Theorem \ref{Sauer}. It is worth mentioning that this operation is also known as \emph{shifting} in contexts such as \cite{A1988} and \cite{A2013}. 
Formally, given a family $\mathcal{F} \subseteq 2^{[n]}$ and a vertex $v \in [n]$, the squash operation at $v$ transforms $\mathcal{F}$ into the following new family:
$$S_v(\mathcal{F}) = \{F_v : F \in \mathcal{F}\},$$
where each modified set $F_v$ is defined as:
$$F_v =
\begin{cases}
F \setminus \{v\} & \text{if } v \in F \in \mathcal{F} \text{ and } F \setminus \{v\} \notin \mathcal{F}, \\
F & \text{otherwise}.
\end{cases}$$

By definition, it is easy to see that $|S_{v}(\mathcal{F})|=|\mathcal{F}|$, and $\sum_{F\in S_{v}(\mathcal{F})}|F|\leqslant \sum_{F\in \mathcal{F}}|F|$,
where this inequality becomes strict if a change is made in this operation. 
Therefore, we observe that it takes a finite number of squash operations, transforming $\mathcal{F}$ into a family $\tilde{\mathcal{F}}$, which satisfies $S_{v}(\tilde{\mathcal{F}})=\tilde{\mathcal{F}}$ for any $v\in [n]$. 
We call such a family $\tilde{\mathcal{F}}$ a \emph{squashed family} obtained from $\mathcal{F}$. 
The following proposition follows directly from the definition.
\begin{proposition}
A squashed family is hereditary.
\end{proposition}

In fact, in order to get a squashed family from $\mathcal{F}$, it is enough to apply the squash $S_{v}$ once for every $v\in [n]$.
More precisely, let $\mathcal{F}\subseteq 2^{[n]}$ and $\sigma$ be a permutation of $[n]$, then the transformed family $$S_{\sigma}(\mathcal{F}):=S_{\sigma(n)}(S_{\sigma(n-1)}(\dots (S_{\sigma(1)}(\mathcal{F}))\dots))$$ is a squashed family, thus it is a hereditary family. 
It is worth noting that $S_{\sigma}(\mathcal{F})$ may vary depending on the choice of the permutation $\sigma$.

The following propositon reveals a close relation between squash operations and trace of sets.
\begin{proposition}[Frankl~\cite{F1983}]\label{squash trace}
Let $\mathcal{F}\subseteq 2^{[n]}$ and $v\in [n]$. For every $T\subseteq [n]$, we have
\[
|S_{v}(\mathcal{F})_{\mid T}|\leqslant |\mathcal{F}_{\mid T}|.
\]
\end{proposition}
\begin{proof}
    If $v\notin T$, we obviously have $S_{v}(\mathcal{F})_{\mid T}=\mathcal{F}_{\mid T}$. We now assume that $v\in T$. For any subset $A\subseteq T\setminus \{v\}$, we claim that
    \[
    |S_{v}(\mathcal{F})_{\mid T}\cap \{A,A\cup \{v\}\}|\leqslant |\mathcal{F}_{\mid T}\cap \{A,A\cup \{v\}\}|.
    \]
    If $A\notin S_{v}(\mathcal{F})_{\mid T}$ and $A\cup \{v\}\notin S_{v}(\mathcal{F})_{\mid T}$, this inequality is trival. If $A\cup \{v\}\in S_{v}(\mathcal{F})_{\mid T}$, then both $A$ and $A\cup \{v\}$ are members of $S_{v}(\mathcal{F})_{\mid T}\cap \mathcal{F}_{\mid T}$. If $A\in S_{v}(\mathcal{F})_{\mid T}$ and $A\cup \{v\}\notin S_{v}(\mathcal{F})_{\mid T}$, we have either $A\in \mathcal{F}_{\mid T}$ or $A\cup \{v\}\in \mathcal{F}_{\mid T}$. Thus we prove our claim. Then we have
    \[
    |S_{v}(\mathcal{F})_{\mid T}|=\sum_{A\subseteq T\setminus \{v\}}|S_{v}(\mathcal{F})_{\mid T}\cap \{A,A\cup \{v\}\}|\leqslant \sum_{A\subseteq T\setminus \{v\}}|\mathcal{F}_{\mid T}\cap \{A,A\cup \{v\}\}|=|\mathcal{F}_{\mid T}|,
    \]
    which proves the proposition.
\end{proof}

The following proof of Lemma~\ref{1.1}, building on Proposition \ref{squash trace}, is due to Frankl~\cite{F1983}. 
\begin{proof}[Proof of Lemma~\ref{1.1}]
By definition, the first statement obviously implies the second. Therefore, it is enough to show that if the second statement holds, then for any family $\mathcal{F}\subseteq 2^{[n]}$ with $|\mathcal{F}|\geqslant m$, we have $\mathcal{F}\rightarrow (a, b)$.
Suppose the contrary that there exists such a family $\mathcal{F}\subseteq 2^{[n]}$ with $\mathcal{F} \not\rightarrow (a, b)$, and subject to this, we may assume $\mathcal{F}$ minimizes $\sum_{F\in \mathcal{F}}|F|$. 
By assumption, $\mathcal{F}$ is not hereditary. Therefore, there exists $v\in [n]$ such that $\sum_{F\in S_{v}(\mathcal{F})}|F|< \sum_{F\in \mathcal{F}}|F|$. The minimality of $\sum_{F\in \mathcal{F}}|F|$ implies that $S_{v}(\mathcal{F})\rightarrow (a,b)$. Then there exists $T\subseteq [n]$ with $|T|=a$ such that $|S_{v}(\mathcal{F})_{\mid T}|\geqslant b$. By Proposition \ref{squash trace}, we know that
\[
|\mathcal{F}_{\mid T}|\geqslant |S_{v}(\mathcal{F})_{\mid T}|\geqslant b,
\]
which is a contradiction with $\mathcal{F} \not\rightarrow (a,b)$. So we can conclude that $(n, m) \rightarrow (a, b)$.
\end{proof}

\section{Defect Sauer results}\label{sec 3}

The Sauer-Shelah Lemma motivates the study of ``defect Sauer results'' (a term introduced in \cite{BR1995}), which study the minimum size required for a family $\mathcal{F}$ to ensure a trace of size at least $b$ on some $a$-element set.
Expressed through arrow relation terminology, this problem reduces to determining the minimum $m_0$ such that the relation $(n, m) \rightarrow (a, b)$ holds for all $m \geq m_0$ given fixed parameters $a$ and $b$. This is formalized by the following notation.

\begin{definition}
    Let $\M(n,a,b)$ denote the minimum integer $m$ such that $(n,m) \rightarrow (a,b)$.
\end{definition}

This notation allows us to express the Sauer-Shelah Lemma (\cref{Sauer}) concisely as
\begin{theorem}[Sauer-Shelah Lemma]For $k \geqslant 1$,
$$\M(n,k,2^{k})=\sum_{i=0}^{k-1}\binom{n}{i}+1.$$
\end{theorem}

In 1983, Frankl~\cite{F1983} pioneered the study of defect Sauer results by resolving a conjecture of Bondy and Lov\'{a}sz, proving the following theorem.
\begin{theorem}[Frankl~\cite{F1983}]\label{(3,7)}
$$\M(n,3,7)=\left\lfloor \frac{n^2}{4} \right\rfloor + n + 2.$$
\end{theorem}

The proof of this theorem is as follows,
which combines Mantel's theorem on $\text{ex}(n,K_3)$ with the hereditary family reduction of the arrow relation via Lemma~\ref{1.1}.

\begin{proof}[Proof of Theorem \ref{(3,7)}]
We establish the upper bound by contradiction. Suppose $\M(n,3,7) >\lfloor \frac{n^2}{4} \rfloor + n + 2$.
By Lemma~\ref{1.1}, we may assume there exists a hereditary family $\mathcal{F}\subseteq 2^{[n]}$ such that $|\mathcal{F}|=  \lfloor \frac{n^2}{4} \rfloor + n + 2$ and $\mathcal{F}\not \rightarrow (3, 7)$.
As $\mathcal{F}$ is a hereditary family, if there exists set $T\in\mathcal{F}$ and $|T|=3$, then $|\mathcal{F}_{|T}|=|2^T|=8$, a contradiction. Thus we may assume that every set $F$ in the hereditary family $\mathcal{F}$ has $|F|\leqslant 2$. 
Moreover, if $\{\{a,b\},\{a,c\},\{b,c\}\}\subseteq\mathcal{F}$ for some $a,b,c\in [n]$, then $|\mathcal{F}_{|\{a,b,c\}}|\geqslant 7$, contradicting the assumption $\mathcal{F}\not \rightarrow (3, 7)$. Thus the collection of $2$-element sets in $\mathcal{F}$ constitutes a $K_3$-free graph. Applying Mantel's theorem yields that
$$ 2 + n + \left\lfloor \frac{n^2}{4} \right\rfloor=|\mathcal{F}| = \left|\mathcal{F}\cap \binom{[n]}{0}\right| + \left|\mathcal{F}\cap \binom{[n]}{1}\right| + \left|\mathcal{F}\cap \binom{[n]}{2}\right| \leq 1 + n + \mbox{ex}(n,K_3)= 1 + n + \left\lfloor \frac{n^2}{4} \right\rfloor,$$
a contradiction. This  establishes the upper bound $\M(n,3,7)\leq\lfloor \frac{n^2}{4} \rfloor + n + 2$.

For the lower bound, let $\mathcal{F}=\binom{[n]}{0}\cup\binom{[n]}{1}\cup \mathcal{G}$, where $\mathcal{G}\subseteq\binom{[n]}{2}$ denotes the edge set of the Tur\'{a}n graph $T(n,2)$. Then $|\mathcal{F}|=\lfloor \frac{n^2}{4} \rfloor+n+1$ and $\mathcal{F}\not \rightarrow (3, 7)$, which yields $\M(n,3,7)\geqslant\lfloor \frac{n^2}{4} \rfloor + n + 2$.
\end{proof}

Employing similar arguments with a more detailed analysis, one can determine $$\M(n,4,11)=\mbox{ex}(n,K_{4})+n+2=\left\lfloor \frac{n^2}{3} \right\rfloor+n+2.$$
Following these developments, Bollob\'{a}s and Radcliffe~\cite{BR1995} further determined the exact value for $\M(n,4,12)$ in 1995.

\begin{theorem}[Bollob\'{a}s and Radcliffe~\cite{BR1995}]\label{(4,12)}
For the case $n=6$, $\M(6,4,12)=24$. And for $n\neq 6$,
$$ \M(n,4,12)=\binom{n}{2}+n+2.$$
\end{theorem}
In 2023, Frankl and Wang~\cite{FW2024} established the exact value of $\M(n,4,13)$ in the following theorem.
\begin{theorem}[Frankl and Wang~\cite{FW2024}]\label{(4,13)}
    For $n\geqslant 25$,
\[
\M(n,4,13)=1+\left \lfloor \frac{n+5}{3} \right \rfloor\left \lfloor \frac{n+4}{3} \right \rfloor\left \lfloor \frac{n+3}{3} \right \rfloor.
\]
\end{theorem}
Besides, Frankl and Wang~\cite{FW2024} also investigated the value of $\M(n,4,b)$ for $6\leq b \leq 15$.
Combining with other known Tur\'an type problems on graphs and hypergraphs, the following table summarizes the values on $\M(n,a,b)$ for $a\in \{2,3,4\}$ and $1\leqslant b\leqslant 2^a$.
Here, we let $C_3^{+}$ denote the graph formed by attaching a pendant edge to a triangle, $K_4^{(3)}$ denote the complete $3$-uniform hypergraph on $4$ vertices, and $K_4^{(3)-}$ denote the 3-uniform hypergraph obtained from $K_4^{(3)}$ by deleting an edge.

\begin{table}[H]
\centering
\begin{tabular}{|c|c|c|c|} 
\toprule
\hline
 & a=2 & a=3 & a=4 \\
\hline
$b=1 $   & $1$   & $1$ & $1$ \\ \hline
$b=2 $   & $2$   & $2$ & $2$ \\ \hline
$b=3 $   & $3$   & $3$ & $3$ \\ \hline
$b=4 $   & $n+2$ & $4$ & $4$ \\ \hline
$b=5 $   & --& $n+2$ & $5$ \\ \hline
$b=6 $   & --& $\lfloor \frac{3}{2}n  \rfloor+2$ & $n+2$ \\ \hline
$b=7 $   & --& $\lfloor \frac{n^{2}}{4} \rfloor+n+2$ & $n+3$ \\ \hline
$b=8 $   & --& $\binom{n}{2}+n+2$ & $\lfloor \frac{5}{3}n  \rfloor+2$ \\ \hline
$b=9 $   & --& --& $\mbox{ex}(n,\{C_{3}^+,C_{4}\})+n+2$ \\ \hline
$b=10$    & --& --& $\lfloor \frac{n^{2}}{4} \rfloor+n+2$ \\ \hline
$b=11$    & --& --&  $\lfloor \frac{n^2}{3} \rfloor+n+2$ \\ \hline
$b=12$, $n=6$    & --& --& $24$ \\ \hline
$b=12$, $n\neq 6$   & --& --& $\binom{n}{2}+n+2$  \\  \hline
$b=13$, $n\geqslant 25$   & --& --& $\lfloor \frac{n+5}{3}  \rfloor\lfloor \frac{n+4}{3}  \rfloor\lfloor \frac{n+3}{3}  \rfloor+1$ \\ \hline
$b=14$    & --& --& $\operatorname{ex}_3(n,K^{(3)-}_{4})+\binom{n}{2}+n+2$ \\ \hline
$b=15$    & --& --& $\operatorname{ex}_3(n,K^{(3)}_{4})+\binom{n}{2}+n+2$ \\ \hline
$b=16$    & --& --& $ \binom{n}{3}+\binom{n}{2}+n+2$ \\ \hline

\end{tabular}
\caption{Values of \( m(n,a,b) \) for \( a \in \{2,\,3,\,4 \}\) and \( 1 \leqslant b \leqslant 2^a \).}
\end{table}

Note that all values in this table are precisely established except for $m(n,4,9)$, $m(n,4,14)$ and $m(n,4,15)$.
As for $m(n,4,9)$, the exact value of $\mbox{ex}(n,\{C_{3}^+,C_{4}\})$ remains undetermined. 
The {\it Zarankiewicz number} $z(n, C_4)$, extensively studied in extremal graph theory \cite{DHS2013,FS2013,NV2005,MY2023,HMY2023}, is defined as the maximum number of edges in an $n$-vertex $C_4$-free bipartite graph. It is known that $z(n, C_4)=\left(\frac{n}{2}\right)^{3/2}+o(n^{3/2})$, which implies $\mbox{ex}(n,\{C_{3}^+,C_{4}\}) \geqslant z(n, C_4)=\left(\frac{n}{2}\right)^{3/2}+o(n^{3/2})$. Similarly, the classical result $\mbox{ex}(n,C_{4})=\frac{1}{2} n^{3/2}+O(n)$ (see \cite{KST1954}) yields $\mbox{ex}(n,\{C_{3}^+,C_{4}\})\leqslant \mbox{ex}\left(n, C_4\right)=\frac{1}{2} n^{3/2}+O(n)$. Consequently, we obtain the bounds:\[\left(\frac{n}{2}\right)^{3/2}+o(n^{3/2}) \leqslant m(n,4,9)\leqslant \frac{1}{2} n^{3/2}+O(n).\]

Regarding $m(n, 4,14)$ and $m(n, 4,15)$, their exact values depend on the hypergraph Tur\'an problems
$\operatorname{ex}_3(n,K^{(3)-}_{4})$ and $\operatorname{ex}_3(n,K^{(3)}_{4})$,
which are notorious open problems in extremal graph theory. 
For complete $3$-uniform hypergraph, Tur\'{a}n~\cite{T1941} proposed the following conjecture.
\begin{conjecture}[Tur\'{a}n~\cite{T1941}]\label{conj K^3_k}
    For any integer $k\geqslant 4$, we have
    \[
    \pi(K^{(3)}_{k})=1-\left ( \frac{2}{k-1} \right)^{2}.
    \]
\end{conjecture}
In the particular case $k=4$, Conjecture \ref{conj K^3_k} asserts that $\pi(K^{(3)}_{4})=\frac{5}{9}$. 
For its lower bound, there have been several constructions (see, e.g., \cite{B1983,F1988,K1982}) showing $\pi(K^{(3)}_{4})\geq\frac{5}{9}$. 
As for the upper bound, progress has been made in \cite{C1991,CL1999}. 
The best known upper bound $\pi(K_4^{(3)}) \leqslant 0.561666$ was obtained by Razborov~\cite{R2010} using flag algebra.
For $K_4^{(3)-}$, the current best bounds are
$\frac{2}{7} \leqslant \pi(K_4^{(3)-}) \leqslant 0.286889$,
where the lower bound is due to Frankl and F\"{u}redi~\cite{FF1984}, and the upper bound was obtained by Falgas-Ravry and Vaughan~\cite{FV2011}.
Consequently, we have the following bounds:
\[
\frac{2}{7}\times\frac{n^3}{6}+o(n^3)\leqslant m(n,4,14)\leqslant 0.286889\times\frac{n^3}{6}+o(n^3),
\]
\[
\frac{5}{9}\times\frac{n^3}{6}+o(n^3)\leqslant m(n,4,15)\leqslant 0.561666\times\frac{n^3}{6}+o(n^3).
\]

\section{Single-Element Removal}\label{sec 4}

By definition, the arrow relation is inherently transitive: if $(n,m)\rightarrow(a,b)$ and $(a,b)\rightarrow(c,d)$, then $(n,m)\rightarrow(c,d)$. Therefore, it is essential to investigate the special arrow relation $(n,m)\rightarrow(n-1,m^\prime)$. Since single-element removal typically causes bounded trace reduction, the relation should be parameterized by a controlled loss rate $s$ rather than an arbitrary $m^\prime$. This leads to the following refined formulation presented in the book by Frankl and Tokushige \cite{FT2018}.
\begin{problem}[Frankl and Tokushige \cite{FT2018}, Problem 3.8]\label{2.1}
For any non-negative integers $n$ and $s$, determine or estimate the maximum value $m=m(n, s)$ such that $(n, m) \rightarrow(n-1, m-s)$.
\end{problem}
In view of Lemma \ref{1.1}, Problem \ref{2.1} is equivalent to the following problem:
\begin{problem}
    Find the maximum $m=m(n, s)$ such that every hereditary family $\mathcal{F}$ on $n$ vertices with $m$ edges satisfies $\delta(\mathcal{F})\leqslant s$.
\end{problem}

The investigation of Problem~\ref{2.1} initially focused on determining $m(n,s)$ for small values of $s$. Bondy \cite{B1972} and Bollob\'as \cite{L2007} solved the cases of $s=0$ and $s=1$ respectively.
\begin{theorem}[Bondy \cite{B1972}]
Let $n$ be a positive integer. Then
$$m(n,0)=n.$$
\end{theorem}
\begin{theorem}[Bollob\'as \cite{L2007}]
Let $n$ be a positive integer. Then
$$m(n,1)=\left \lceil \frac{3}{2}n \right \rceil.$$
\end{theorem}

In 1983, Frankl~\cite{F1983} proved the first result for an infinite family of cases, specifically for $s = 2^{d-1} - 1$, where $d$ is any positive integer.

\begin{theorem}[Frankl~\cite{F1983}]\label{thm2^(d-1)-1}
Let $n$ and $d$ be positive integers.
$$m(n,2^{d-1}-1)\geq \left \lceil \frac{2^{d}-1}{d}n \right \rceil.$$

 \end{theorem}
The proof of Theorem \ref{thm2^(d-1)-1} relies on the well-known Kruskal-Katona Theorem, named after Kruskal~\cite{K1963} and Katona~\cite{K1968}.
For two finite sets $A,B\subseteq \mathbb{N}^{+}$, we say that $A$ precedes $B$ in the {\it colexicographic order}, which is denoted by $A\prec_{col}B$, if $\max(A\bigtriangleup B)\in B$.
For $m\in \mathbb{N}$, we define $\mathcal{R}(m)$ to be the family containing the first $m$ finite subsets of $\mathbb{N}^{+}$ according to the colexicographic order. For $2^{k}\leqslant m<2^{k+1}$, we have $$\mathcal{R}(m)=2^{[k]}\cup \{R\cup \{k+1\}:R\in \mathcal{R}(m-2^{k})\}. $$
In particular, $\mathcal{R}(0)=\emptyset $ and $\mathcal{R}(2^{k})=2^{[k]}$.
The following theorem due to Katona \cite{K1978} is a generalization of the Kruskal-Katona theorem.
\begin{theorem}[Katona \cite{K1978}] \label{Kruskal-Katona theorem}
Let $f:\mathbb{N}\to \mathbb{R}$ be a monotone non-increasing function and let $\mathcal{F}$ be a hereditary family with $|\mathcal{F}|=m$. Then
\begin{equation}
\sum_{F\in \mathcal{F}}f(|F|)\ge \sum_{R\in \mathcal{R}(m)}f(|R|).\notag
\end{equation}
\label{katona}
\end{theorem}

We refer readers to Chapter 8.2 of \cite{GP2018} for a proof of Theorem \ref{Kruskal-Katona theorem}. We now present the proof of Theorem \ref{thm2^(d-1)-1}, originally established by Frankl~\cite{F1983}.

\begin{proof}[Proof of Theorem \ref{thm2^(d-1)-1}]
By Lemma \ref{1.1} and the definition of arrow relation, it suffices to show that a hereditary family $\mathcal{F}\subseteq 2^{[n]}$ with $\delta(\mathcal{F})\geqslant 2^{d-1}$ must satisfy $|\mathcal{F}|>\left \lceil \frac{2^{d}-1}{d}n \right \rceil$. 
For any $x\in [n]$, let $\mathcal{F}(x)=\{F\setminus \{x\}:x\in F\in \mathcal{F}\}$ denote the link of $x$. 
Then $|\mathcal{F}(x)|\geq2^{d-1}$ for any $x\in [n]$. 
It is easy to see that the fact that $\mathcal{F}$ is hereditary implies that $\mathcal{F}(x)$ is also hereditary. 
Applying Theorem~\ref{katona} for $\mathcal{F}(x)$ and $f(k)=\frac{1}{k+1}$, we can obtain that for any $x\in [n]$,
\[
\sum_{H\in \mathcal{F}(x)}\frac{1}{|H|+1}\geqslant \sum_{R\in \mathcal{R}(|\mathcal{F}(x)|)}\frac{1}{|R|+1}\geqslant \sum_{R\in 2^{[d-1]}}\frac{1}{|R|+1}=\sum_{i=0}^{d-1}\frac{\binom{d-1}{i}}{i+1}=\frac{2^{d}-1}{d}.
\]
Thus by the method of double-counting, we obtain
\[
|\mathcal{F}\setminus \{\emptyset \}|=\sum_{\emptyset \ne F\in \mathcal{F}}\sum_{x\in F}\frac{1}{|F|}=\sum_{x\in [n]}\sum_{x\in F\in \mathcal{F}}\frac{1}{|F|}=\sum_{x\in [n]}\sum_{H\in \mathcal{F}(x)}\frac{1}{|H|+1}\geqslant \frac{2^{d}-1}{d}n.
\]
This implies that $|\mathcal{F}|\geqslant \frac{2^{d}-1}{d}n+1 >\left \lceil \frac{2^{d}-1}{d}n \right \rceil$, as desired.
\end{proof}

The upper bound of $m(n,2^{d-1}-1)$ when $d\mid n$ is given by the following general construction:
\begin{construction}\label{Construction}
Let $d$ and $c$ be positive integers with $1\leqslant c\leqslant d-1$. Assume that $n=dk$ for some positive integer $k$.
Let $U_{1},...,U_{k}$ form a partition of $[n]$ into sets of size $d$.
For every $i\in [k]$, arbitrarily pick a family $\mathcal{G}_{i}\subseteq 2^{U_{i}}\setminus \{\emptyset\}$ with $|\mathcal{G}_{i}|=c-1$.
Define
\[
\mathcal{F}(n,d,c) := \left\{ F\subseteq [n] : F\in 2^{U_{i}}\setminus \mathcal{G}_{i} \mbox{ for some $i\in [k]$}\right\}.
\]

For $\mathcal{F}=\mathcal{F}(n,d,c)$ in this construction, one can verify that $|\mathcal{F}|=\frac{2^{d}-c}{d}n+1$ and for any $T\in \binom{[n]}{n-1}$, we have $|\mathcal{F}_{\mid T}|<|\mathcal{F}|-2^{d-1}+c$. This implies that $m(n,2^{d-1}-c)\leqslant \frac{2^{d}-c}{d}n$ whenever $d\mid n$. 
In particular, by letting $c=1$, we now can obtain
\[
m(n,2^{d-1}-1)= \frac{2^{d}-1}{d}n \mbox{ for } d\mid n.
\]
\end{construction}

Since tight constructions matching the lower bound exist only when $d\mid n$, determining $m(n,s)$ for arbitrary values of $n$ becomes less tractable. This led researchers to shift focus toward the asymptotic behavior of $m(n,s)$ instead. In 1994, Watanabe and Frankl \cite{FW1994} established that for any $s\geqslant 0$, the limit $\lim_{n\to \infty}\frac{m(n,s)}{n}$ exists, and they introduced the notation $m(s)$ for this limit.

\begin{theorem}[Watanabe and Frankl \cite{FW1994}]\label{thmm(s)}
For any $s\geqslant 0$, the limit $\lim_{n\to \infty}\frac{m(n,s)}{n}$ exists.
\end{theorem}
\begin{definition}[Watanabe and Frankl \cite{FW1994}]
    Let $m(s)=\lim_{n\to \infty}\frac{m(n,s)}{n}$ denote the limit.
\end{definition}
Under this notation, 
the above results of Bondy \cite{B1972}, Bollob\'as \cite{L2007}, and Frankl~\cite{F1983} can be restated as $m(0)=1$, $m(1)=\frac{3}{2}$, and $m(2^{d-1}-1)=\frac{2^{d}-1}{d}$ for $d\geq 1$, respectively.
Watanabe and Frankl \cite{FW1994} further determined the values of $m(2^{d-1}-2)$ and $m(2^{d-1})$ for $d\geq 1$ as follows.
\begin{theorem}[Watanabe and Frankl \cite{FW1994}]\label{thm2^(d-1)-2}
Let $n$ and $d$ be positive integers. Then
$$m(2^{d-1}-2)=\frac{2^{d}-2}{d}
\ \ \textit{ and } \ \ m(2^{d-1})=\frac{2^{d}-1}{d}+\frac{1}{2}.$$
\end{theorem}
The lower bounds in both cases are established using arguments similar to those in Theorem \ref{thm2^(d-1)-1}. For the upper bounds, $m(2^{d-1}-2)$ is derived from $\mathcal{F}(n,d,2)$ in Construction \ref{Construction}, while $m(2^{d-1})$ is obtained from the following construction:
\begin{construction}
Let $d$ and $k$ be positive integers and let $n=2dk$. Define
\[
\mathcal{F}=\{ \{j,j+dk\}:j\in [dk] \}\cup \bigcup_{i=0}^{2k-1}2^{[id+1:id+d]}.
\]
One can verify that $|\mathcal{F}|=\left(\frac{2^{d}-1}{d}+\frac{1}{2}\right)n+1$ and for any $T\in\binom{[n]}{n-1}$, we have $|\mathcal{F}_{\mid T}|<|\mathcal{F}|-2^{d-1}$. This demonstrates that $m(2^{d-1})\leq\frac{2^{d}-1}{d}+\frac{1}{2}$ for $d\geq 1$.
\end{construction}
Beyond these precise formulas for $m(s)$ with $s\in \{2^{d-1}, 2^{d-1}-1, 2^{d-1}-2\}$, researchers have also determined the exact values for several specific cases. In 1991, Watanabe \cite{W1991} established that $m(5)=\frac{13}{4}$. Further advancing this line of inquiry, Watanabe and Frankl \cite{FW1994} proved $m(9)=\frac{65}{14}$, while Watanabe \cite{W1995} later determined that $m(10)=5$ and $m(13)=\frac{29}{5}$.

A recent result of Piga and Sch\"{u}lke \cite{PS2021} determined $m(12)$ and $m(2^{d-1}-c)$ for a range of $c$.
\begin{theorem}[Piga and Sch\"{u}lke \cite{PS2021}]\label{thmpiga2}
$$m(12)=\frac{28}{5}.$$
\end{theorem}
\begin{theorem}[Piga and Sch\"{u}lke \cite{PS2021}]\label{thmpiga1}
Let $1\leq c\leq \frac{d}{4}$. Then,
\[
m(2^{d-1}-c)= \frac{2^{d}-c}{d}.
\]
\end{theorem}
Piga and Sch\"{u}lke \cite{PS2021} also provided the following construction.
\begin{construction}
Let $d\geq 3$ and $k$ be positive integers and let $n=2dk$. Define
\[
\mathcal{F}=\{[jd-d+1:jd-1]:j\in[2k]\}\cup\{ \{2dt-d,2dt\}:t\in[k]\}\cup \bigcup_{i=0}^{2k-1}\binom{[id+1:(i+1)d]}{\leq d-2}.
\]
One can verify that $|\mathcal{F}|=\frac{2^{d}-d-\frac{1}{2}}{d} n+1$ and for any $T\in\binom{[n]}{n-1}$, we have $|\mathcal{F}_{\mid T}|<|\mathcal{F}|-2^{d-1}+d$. This demonstrates that $m(2^{d-1}-d)\leqslant \frac{2^{d}-d-\frac{1}{2}}{d}$ for $d\geq 3$.
\end{construction}

In 2024, Li, Ma and Rong \cite{LMR2024} investigated this problem, by determining the precise values of $m(11)$ as well as $m(2^{d-1}-c)$ whenever $1\leqslant c\leqslant d$ and $d\geq 50$.
The result $m(11)=\frac{53}{10}$ confirmed a conjecture of Watanabe and Frankl \cite{FW1994} from 1994, while the latter one solved a problem posted by Piga and Sch\"{u}lke \cite{PS2021}. 

\begin{theorem}[Li, Ma and Rong \cite{LMR2024}]
$$m(11)=\frac{53}{10}.$$
\end{theorem}
\begin{theorem}[Li, Ma and Rong \cite{LMR2024}]\label{thmour}
Let $d\geqslant 50$ and $1\leqslant c\leqslant d-1$. Then
$$ m(2^{d-1}-c)=\frac{2^{d}-c}{d} \ \ \textit{ and } \ \  m(2^{d-1}-d)=\frac{2^{d}-d-\frac{1}{2}}{d} .$$
\end{theorem}
In their proof of Theorem \ref{thmour}, they partitioned the ground set $[n]$ into several ``piles'', which are neighborhoods of vertices with small weight. The proof focused on estimating the average weight in ``intersecting piles'' and ``non-intersecting piles''.
The result $m(11)=\frac{53}{10}$ was proved by performing a perturbation on the weight function.
With the determination of $m(11)$,  all the values of $m(s)$ for $s\leqslant 16$ are now known, and for convenience we list them all in the following table.
\smallskip
	
	\begin{center}
        \renewcommand{\arraystretch}{1.5}
		\begin{tabular}{|c|*{17}{C{0.5cm}|}}
        \hline
			$s=$ & $0$ & $1$ & $2$ & $3$ & $4$ & $5$ & $6$ & $7$ & $8$ & $9$ & $10$ & $11$ & $12$
			& $13$ & $14$ & $15$ & $16$ \\ \hline
			$m(s)=$ & $1$ & $\frac 32$ & $2$ & $\frac 73$ & $\frac{17}6$ & $\frac{13}4$
			& $\frac 72$ & $\frac{15}4$ & $\frac{17}4$ & $\frac{65}{14}$ & $5$ & $\frac{53}{10}$
			& $\frac{28}5$ & $\frac{29}5$ & $6$ & $\frac{31}5$ & $\frac{67}{10}$ \\
        \hline
		\end{tabular}
	\end{center}	
\smallskip

Very recently, Reiher and Sch\"{u}lke \cite{RS2025} independently obtained similar results. 
In their proof, they grouped vertices together into ``conglomerates'', which are $d$-element sets and contain most of hyperedges at some vertices. Then they managed to show that the average weight in each conglomerate is sufficiently large and any two conglomerates intersect in at most one vertex.
Using their approach, they proved that $m(11) = \frac{53}{10}$ and established the following result on $m(2^{d-1} - c)$ for $1 \leq c \leq d-1$, which closes the gap for $6 \leq d \leq 49$ in Theorem \ref{thmour}.
\begin{theorem}[Reiher and Sch\"{u}lke \cite{RS2025}]\label{thm:main}
Let $d$ and $c$ be positive integers with $1\leqslant c\leqslant d-1$.
Then $$m(2^{d-1}-c)=\frac{2^{d}-c}{d}. $$
\end{theorem}
These aforementioned theorems determine the value of $m(s)$ when $s$ is slightly smaller than a power of $2$. The problem of determining $m(s)$ remains open when $s$ lies further away from powers of two, or when $s$ slightly exceeds powers of two.

\section{Trace Functions}\label{sec 5}

The framework of arrow relation $(n, m) \rightarrow (a, b)$ naturally extends to a variety of extremal problems in combinatorial set theory. In Section \ref{sec 3}, we examined defect Sauer results, which focus on determining the minimum $m$ attainable under fixed parameters $n$, $a$, and $b$. 
This section focuses on a complementary extremal problem of equal significance: for prescribed values of $n$, $m$, and $a$, what is the maximum $b$ such that the relation $(n, m) \rightarrow (a, b)$ holds?

In order to formalize the optimization of $b$ in $(n, m) \rightarrow (a, b)$, let us introduce the following definition.
\begin{definition}
    For a family $\mathcal{F}\subseteq 2^{[n]}$ and $1\leq a\leq n$, the \emph{trace function} is defined as $\Tr(\mathcal{F},a)=\max_{T\in \binom{[n]}{a}}|\mathcal{F}_{|T}|$. Furthermore, let $\Tr(n,m,a)$ denote the maximum size of traces on $a$-element sets guaranteed in any families of size $m$, that is
\[
\Tr(n,m,a)=\min_{\substack{\mathcal{F}\subseteq 2^{[n]} \\ |\mathcal{F}|=m}}\Tr(\mathcal{F},a)=\min_{\substack{\mathcal{F}\subseteq 2^{[n]} \\ |\mathcal{F}|=m}}\max_{T\in \binom{[n]}{a}}|\mathcal{F}_{|T}|.
\]
\end{definition}
Combining with the arrow relation, we have the following proposition.
\begin{proposition}
    For positive integers $n$, $m$ and $a$, $\Tr(n,m,a)$ is precisely the maximum $b$ for which the arrow relation $(n, m) \rightarrow (a, b)$ holds.
\end{proposition}

A fundamental contribution to this direction was established by Bollobás and Radcliffe~\cite{BR1995}, who determined a general lower bound for $\Tr(n,m,a)$:
\begin{theorem}[Bollobás and Radcliffe~\cite{BR1995}]\label{vol}
For any positive integers $a\leq n$ and $1\leqslant m\leqslant 2^{n}$, we have
\[
\Tr(n,m,a)\geqslant m^{a/n}.
\]
\end{theorem}

Although Theorem~\ref{vol} is concise and useful, it does not necessarily give the best possible bound for $b$. Specifically, let $m=n^{r}$ and $a=\alpha n$, Theorem~\ref{vol} shows that $(n,n^{r})\rightarrow (\alpha n,n^{\alpha r})$. The following result of Bollob\'{a}s and Radcliffe~\cite{BR1995} improves this bound by using probabilistic method.
Let $H(x) = -x \log_{2} x - (1 - x) \log_{2}(1 - x)$ denote the binary entropy function.
\begin{theorem}[Bollob\'{a}s and Radcliffe~\cite{BR1995}]\label{alpha n}
    For any integer $r\geqslant 2$ and real number $0<\alpha <1$, we have
\[
\Tr(n,n^{r},\alpha n)\geqslant (1-o(1))n^{\lambda r},
\]
where we set $\lambda_{0}=\log_{2}(1+\alpha)$, and $\lambda$ is defined as
\[
\lambda =\begin{cases}
\hfil \frac{\lambda_{0}}{H(\lambda_{0})}  &\mbox{ if } \alpha \in (0,\sqrt{2}-1); \\
\hfil \lambda_{0}  &\mbox{ if } \alpha \in [\sqrt{2}-1,1).
\end{cases}
\]
\end{theorem}

Theorem~\ref{alpha n} states that $\Tr(n, n^{r}, \alpha n) = \Omega(n^{\lambda r})$. Additionally, there are several influential results on estimating trace functions. 
One notable example is the following classical result by Kahn, Kalai and Linial~\cite{KKL1988}.
\begin{theorem}[Kahn, Kalai and Linial~\cite{KKL1988}]
    For any integer $n\geq 1$ and any $0<\alpha, \beta<1$, we have
    \[
    \Tr(n,\beta 2^{n},\alpha n)\geqslant (1-n^{-c})2^{\alpha n},
    \]
    where $c=c(\alpha,\beta)$ depends only on $\alpha$ and $\beta$.
\end{theorem}

To pursue the upper bound, it can be argued that among families $\mathcal{F} \subseteq 2^{[n]}$ of size $m$, the extremal family $\mathcal{F}$ satisfying $\Tr(\mathcal{F}, k) = \Tr(n, m, k)$ should exhibit some form of symmetry. 
Thus, it is natural to conjecture that the extremal family includes all sets up to a certain size. 
However, Bollob\'{a}s and Radcliffe~\cite{BR1995} proved the following theorem, demonstrating that this is not the case.

\begin{theorem}[Bollob\'{a}s and Radcliffe~\cite{BR1995}]\label{upper bound}
    For any integer $r\geqslant 2$ and real number $0<\varepsilon <\frac{1}{2}$, there exists an $n_{0}=n_{0}(r,\varepsilon)$, such that for any $n\geqslant n_{0}$, there exists a hereditary family $\mathcal{H}\subseteq 2^{[n]}$ with $|\mathcal{H}|\geqslant \sum_{i=0}^{r}\binom{n}{i}$ and \[\Tr(\mathcal{H},n/2)\leqslant \sum_{i=0}^{r}\binom{n/2}{i}-(1-\varepsilon)2^{-r}\binom{n}{r}. \]
\end{theorem}
As a corollary of Theorem~\ref{upper bound}, we can derive $\Tr(n,\sum_{i=0}^{r}\binom{n}{i},\frac{n}{2})=o(n^{r})$ for any integer $r\geqslant 2$.
Moreover, if we set $r=2$ and $\alpha =1/2$, then Theorems~\ref{alpha n} and \ref{upper bound} bound $\Tr(n,n^{2},n/2)$ as follows:
\[
\Omega(n^{1.1699\dots})=\Omega(n^{2\log_{2}\frac{3}{2}})\leqslant \Tr(n,n^{2},n/2)\leqslant o(n^{2}).
\]
Throughout this example, the gap between the lower and upper bounds is evident.
Recently, Alon, Moshkovitz and Solomon~\cite{AMS2019} determined the value of $\Tr(n,n^{r},\alpha n)$ up to logarithmic factors whenever $r$ and $\alpha$ are constants.
They introduced a new version of the Kruskal-Katona Theorem, referred to as the Sparse Kruskal-Katona Theorem, and derived the following theorem from it.
\begin{theorem}[Alon, Moshkovitz and Solomon~\cite{AMS2019}]\label{alpha n value}
    For any integer $r\geqslant 1$ and real number $\alpha \in (0,1]$, if $r,\alpha^{-1}\leqslant n^{o(1)}$ then $\Tr(n,n^{r},\alpha n)=n^{\mu(1-o(1))}$ where
\[
\mu=\mu(r,\alpha)=\frac{r+1-\log(1+\alpha)}{2-\log(1+\alpha)}.
\]
Moreover, if $r=O(1)$ and $\alpha^{-1}\leqslant (\log(n))^{O(1)}$, then $\Tr(n,n^{r},\alpha n)=\tilde{\Theta}(n^{\mu})$.\footnote{Here, all logarithms are in base $2$, and we use the following standard notation: for two functions $f(n)$ and $g(n)$, we write $f=\tilde{\Theta}(g)$ if $f=\Theta(g\log^{c}(g))$ for some absolute constant $c>0$.}
\end{theorem}
In particular, by setting $r=2$ and $\alpha =1/2$, Theorem~\ref{alpha n value} implies that
\[
\Tr(n,n^{2},n/2)=\tilde{\Theta}(n^{1+\frac{1}{3-\log_{2}3}})=\tilde{\Theta}(n^{1.706695\dots}).
\]

\section{Extremal Families}\label{6 Extremal families}
Characterizing combinatorial structures that optimize the given parameters is a central problem in extremal combinatorics.
This section focuses on structural results on the largest families with trace constraints.  
Before proceeding, we introduce some basic definitions.

\begin{definition}\label{tr F}
    A family $\mathcal{F}\subseteq 2^{[n]}$ is said to \emph{trace} (or \emph{shatter}) a set $T\subseteq [n]$ if $\mathcal{F}_{|T}=2^{T}$. We denote the collection of traced sets by $tr(\mathcal{F})=\{T: \mathcal{F} \mbox{ traces }T\}$.
\end{definition}

\begin{definition}
    The \emph{Vapnik-Chervonenkis dimension} (or \emph{VC-dimension}) of a family $\mathcal{F}\subseteq 2^{[n]}$, denoted by
    $\VC(\mathcal{F})$, is defined as the maximum cardinality among sets in $tr(\mathcal{F})$, that is
    \[
    \VC(\mathcal{F})=\max_{T\in tr(\mathcal{F})}|T|.
    \]
\end{definition}
With these notions, the Sauer-Shelah lemma can be rewritten as following.
\begin{theorem}[Sauer~\cite{S1972}, Perles and Shelah~\cite{SP1972}, and Vapnik and Chervonenkis~\cite{V1968,VC2015}]\label{Thm:SS}
    For any family $\mathcal{F}\subseteq 2^{[n]}$, the inequality $\VC(\mathcal{F})\geqslant k+1$ holds whenever $|\mathcal{F}|\geqslant 1+\sum_{i=0}^{k}\binom{n}{i}$.
\end{theorem}

The subsequent theorem from Pajor~\cite{P1985} strengthens this result.
\begin{theorem}[Pajor~\cite{P1985}]\label{Pajor}
    For any family $\mathcal{F}\subseteq 2^{[n]}$, the inequality $|\mathcal{F}|\leqslant |tr(\mathcal{F})|$ holds.
\end{theorem}
In other words, a set system $\mathcal{F}$ must trace at least $|\mathcal{F}|$ distinct sets. 
It is evident to observe that Theorem~\ref{Pajor} directly implies Theorem~\ref{Thm:SS}, as there are exactly $\sum_{i=0}^{k} \binom{n}{i}$ subsets of $[n]$ with cardinality at most $k$.
A significant direction of research focuses on the extremal set systems addressed in Theorem~\ref{Thm:SS} and Theorem~\ref{Pajor}. 
Specifically, the former concerns the characterization of maximum families that trace no $(k+1)$-sets, while the latter deals with the characterization of families $\mathcal{F}$ satisfying $|\text{tr}(\mathcal{F})| = |\mathcal{F}|$. 
For the former, Frankl~\cite{F1983} and Dudley~\cite{D1985} characterized families $\mathcal{F} \subseteq 2^{[n]}$ of size $n+1$ with $\VC(\mathcal{F}) = 1$. 
In comparison, the latter has been the subject of more extensive study in the literature.

Similar to Definition \ref{tr F}, Bollob\'{a}s and Radcliffe~\cite{BR1995} proposed the following definition.
\begin{definition}
    We say that a family $\mathcal{F}\subseteq 2^{[n]}$ \emph{strongly traces} a set $T\subseteq [n]$, if there exists a set $S\subseteq [n]$ disjoint from $T$ such that $\mathcal{F}\supseteq \{A\cup S: A\in 2^{T}\}$. The set $S$ is called a \emph{support} of $T$. The family of all sets strongly traced by $\mathcal{F}$ is denoted by $str(\mathcal{F})$.
\end{definition}
By definition, it is easy to see that $str(\mathcal{F})\subseteq tr(\mathcal{F})$ for any family $\mathcal{F}\subseteq 2^{[n]}$. Bollob\'{a}s, Leader, and Radcliffe~\cite{BLR1989} proved the following theorem, which turns out to be equivalent to Theorem~\ref{Pajor}.
\begin{theorem}[Bollob\'{a}s, Leader and Radcliffe~\cite{BLR1989}]
    Any family $\mathcal{F}\subseteq 2^{[n]}$ satisfies $|\mathcal{F}|\geqslant |str(\mathcal{F})|$.
\end{theorem}
This theorem states that a family $\mathcal{F}$ can strongly trace at most $|\mathcal{F}|$ sets. 
For this reason, it is often referred to as the {\it reverse Sauer inequality}. 
In fact, there is a certain symmetry between $\text{tr}(\mathcal{F})$ and $\text{str}(\mathcal{F})$. Specifically, the sets that $\mathcal{F}$ fails to trace are precisely the complements of the sets that $2^{[n]} \setminus \mathcal{F}$ strongly traces. To formalize this observation, we state it as a theorem.
\begin{theorem}[see \cite{BR1995}]\label{tr=str}
    For any family $\mathcal{F}\subseteq 2^{[n]}$, we have
    \[
    2^{[n]}\setminus tr(\mathcal{F})=\{T^{c}: T\in str(2^{[n]}\setminus \mathcal{F})\}.
    \]
\end{theorem}
The advantage of Theorem~\ref{tr=str} is that it allows us to apply the various criteria from \cite{BLR1989} for the condition $|\text{tr}(\mathcal{F})| = |\mathcal{F}|$ to hold, as will be discussed later.
Although Theorem~\ref{tr=str} shows that $|tr(\mathcal{F})|\geqslant |\mathcal{F}|$ and $|str(\mathcal{F})|\leqslant |\mathcal{F}|$ are equivalent, the equivalence between $|tr(\mathcal{F})|=|\mathcal{F}|$ and $|str(\mathcal{F})|=|\mathcal{F}|$ is not trivial.
Recall that by applying the squash operation once for every $i\in [n]$, we can get a hereditary family $S_{\sigma}(\mathcal{F})$, where $\sigma$ is a permutation of $[n]$ and $S_{\sigma}(\mathcal{F})=S_{\sigma(n)}(S_{\sigma(n-1)}(\dots (S_{\sigma(1)}(\mathcal{F}))\dots))$. 
In general, $S_{\sigma}(\mathcal{F})$ is not identical for every permutation $\sigma$. 
The following result shows that the conditions for equality in Sauer's inequality and the reverse Sauer inequality are equivalent.
\begin{theorem}[Bollob\'{a}s and Radcliffe~\cite{BR1995}]\label{tr equ str}
    For any family $\mathcal{F}\subseteq 2^{[n]}$, the following four properties are equivalent:
\begin{itemize}
     \item[(1)] $|tr(\mathcal{F})|=|\mathcal{F}|$.
     \item[(2)] $|str(2^{[n]}\setminus \mathcal{F})|=2^{n}-|\mathcal{F}|$.
     \item[(3)] There is a unique hereditary family $\tilde{\mathcal{F}}$ that can be achieved from $\mathcal{F}$ by squash operations. In other words, $S_{\sigma}(\mathcal{F})=S_{\pi}(\mathcal{F})$ for any permutations $\sigma, \pi$ of $[n]$.
     \item[(4)] $|str(\mathcal{F})|=|\mathcal{F}|$.
\end{itemize}
\end{theorem}
Fix sets $J\subseteq I\subseteq [n]$. Following \cite{BLR1989}, we call a subfamily of the form $\mathcal{C}=\{F\in \mathcal{F}:F\cap I=J\}$ a \emph{chunk} of $\mathcal{F}$. A chunk $\mathcal{C}$ is \emph{self-complementary} if $\{C\bigtriangleup I^{c}:C\in \mathcal{C}\}=\mathcal{C}$. In other words, having fixed the behaviour inside $I$, we check and see whether on $I^{c}$ the chunk looks like a self-complementary family. A chunk $\mathcal{C}$ is called \emph{trivial} if it is either empty or of maximal size, that is $\mathcal{C}=\{J\cup F:F\subseteq I^{c}\}$. For any $A,B\subseteq [n]$, we call $\{F\in \mathcal{F}:A\cap B\subseteq F\subseteq A\cup B\}$ the chunk of $\mathcal{F}$ they span. In particular, this is a chunk of $\mathcal{F}$, since
\[
\{F\in \mathcal{F}:A\cap B\subseteq F\subseteq A\cup B\}=\{F\in \mathcal{F}:F\cap (A\bigtriangleup B)^{c}=A\cap B\}.
\]
Let $\mathcal{F}\subseteq 2^{[n]}$ and $I\subseteq [n]$. 
We write $\mathcal{F}(I)=\{F\in \mathcal{F}:F\subseteq I^{c} \mbox{ and } F\cup J\in \mathcal{F} \mbox{ for any } J\subseteq I\}$. Let $G_{\mathcal{F}}$ denote the inclusion graph of $\mathcal{F}\subseteq 2^{[n]}$, that is the graph with vertex set $\mathcal{F}$ and edge set $\{(F,F')\in \mathcal{F}\times \mathcal{F}:|F\bigtriangleup F'|=1\}$. We say that $\mathcal{F}$ is connected if $G_{\mathcal{F}}$ is connected. The following are several criteria given in \cite{BLR1989} for the family $\mathcal{F}$ with $|tr(\mathcal{F})|=|\mathcal{F}|$.
\begin{theorem}[Bollob\'{a}s, Leader and Radcliffe~\cite{BLR1989}]
    Suppose that $\mathcal{F}\subseteq 2^{[n]}$.
    \begin{itemize}
     \item[(1)] If $|tr(\mathcal{F})|=|\mathcal{F}|$, then for any chunk $\mathcal{C}$ of $\mathcal{F}$, $|tr(\mathcal{C})|=|\mathcal{C}|$.
     \item[(2)] $|tr(\mathcal{F})|=|\mathcal{F}|$ if and only if $\mathcal{F}(I)$ is connected for every $I\subseteq [n]$.
     \item[(3)] $|tr(\mathcal{F})|=|\mathcal{F}|$ if and only if every chunk of every $\mathcal{F}(I)$ is connected. In particular, if $I\subseteq [n]$ and $|tr(\mathcal{F})|=|\mathcal{F}|$, then any two elements of $\mathcal{F}(I)$ can be connected in the chunk of $\mathcal{F}(I)$ they span.
\end{itemize}
\end{theorem}
Furthermore, Bollob\'{a}s and Radcliffe \cite{BR1995} proved the following theorem.
\begin{theorem}[Bollob\'{a}s and Radcliffe~\cite{BR1995}]
$|tr(\mathcal{F})|=|\mathcal{F}|$ if and only if $\mathcal{F}$ contains no non-trivial self-complementary chunks.
\end{theorem}

We can also find different types of characterization of families with $|tr(\mathcal{F})|=|\mathcal{F}|$ in \cite{MR2012, KM2013, MR2014, KM2020}. 
For instance, in \cite{MR2012}, M\'{e}sz\'{a}ros and R\'{o}nyai investigated families of VC dimension at most $1$ with $|tr(\mathcal{F})|=|\mathcal{F}|$ from the view of the inclusion graphs of $\mathcal{F}$. In \cite{MR2014}, M\'{e}sz\'{a}ros and R\'{o}nyai characterized families of VC dimension $2$ with $|tr(\mathcal{F})|=|\mathcal{F}|$ in terms of their inclusion graphs. In \cite{KM2020}, Kusch and M\'{e}sz\'{a}ros discussed an approach to study the extremal families using Sperner families.

In general, we have $|str(\mathcal{F})|\leqslant |\mathcal{F}|\leqslant |tr(\mathcal{F})|$ for any $\mathcal{F}\subseteq 2^{[n]}$. There is a version of shattering that always results in equality, which is called \emph{order-shattering} and is introduced by Anstee, R\'{o}nyai and Sali~\cite{ARS2002}. They define this concept in an inductive way on the size of $T$. Note that this is just the set version of $C(s)$ given in \cite{AA1995}. For $T=\emptyset$, all we need is a single set from $\mathcal{F}$. For $T=\{t_{1}, t_{2},\dots ,t_{|T|}\}$ with $|T|\geqslant 1$ and $t_{1}<t_{2}<\dots <t_{|T|}$, we say that T is order-shattered by $\mathcal{F}$ if there are $2^{|T|}$ sets from $\mathcal{F}$ divided into two families $\tilde{\mathcal{F}}_{0}$ and $\tilde{\mathcal{F}}_{1}$ so that if we define $T^{\prime}=\{t_{|T|}+1, t_{|T|}+2, \dots , n\}$ (if $t_{|T|}=n$, then $T^{\prime}=\emptyset$) we have that 
\[
\begin{cases}
    T^{\prime}\cap C=T^{\prime}\cap D, \\
    t_{|T|}\notin C, \\
    t_{|T|}\in D
\end{cases}
\]
for all $C\in \tilde{\mathcal{F}}_{0}$, $D\in \tilde{\mathcal{F}}_{1}$ and both $\tilde{\mathcal{F}}_{0}$ and $\tilde{\mathcal{F}}_{1}$ individually order-shatter $T\setminus \{t_{|T|}\}$. We define 
\[
\operatorname{osh}(\mathcal{F})=\{T\subseteq [n]: \mathcal{F} \mbox{ order-shatters } T\}. 
\]
Note that $\operatorname{osh}(\mathcal{F})$ is a hereditary family and $\operatorname{osh}(\mathcal{F})\subseteq tr(\mathcal{F})$. Anstee, R\'{o}nyai and Sali~\cite{ARS2002} proved that $\mathcal{F}$ order-shatters exactly $|\mathcal{F}|$ sets. 
\begin{theorem}[Anstee, R\'{o}nyai and Sali~\cite{ARS2002}]
    Let $\mathcal{F}\subseteq 2^{[n]}$. Then we have 
    \[
    |\operatorname{osh}(\mathcal{F})|=|\mathcal{F}|. 
    \]
\end{theorem}

Another intriguing problem is to determine the number of extremal families. 
Let $f(n,k)$ denote the number of families $\mathcal{F}\subseteq 2^{[n]}$ of VC dimension $k$ with $|\mathcal{F}|=\sum_{i=0}^{k}\binom{n}{i}$. Frankl~\cite{F1989} raised the problem of estimating $f(n,k)$ and showed that $2^{\binom{n-1}{k}}\leqslant f(n,k)\leqslant 2^{n\binom{n-1}{k}}$. 
Better upper and lower bounds were obtained by Alon, Moran, and Yehudayoff~\cite{AMY2016}. 
They showed that for any integer $k\geqslant 2$, as $n\to \infty$, we have
\[
n^{(1+o(1))\frac{1}{k+1}\binom{n}{k}}\leqslant f(n,k)\leqslant n^{(1+o(1))\binom{n}{k}}.
\]
Recently, Balogh, M\'{e}sz\'{a}ros and Wagner~\cite{BMW2018} closed the gap and showed that the upper bound is asymptotically tight, even if we allow $k$ to grow as $k=n^{o(1)}$.
\begin{theorem}[Balogh, M\'{e}sz\'{a}ros and Wagner~\cite{BMW2018}]
Let $k=n^{o(1)}$. Then $f(n,k)= n^{(1+o(1))\binom{n}{k}}$ as $n\to \infty$.
\end{theorem}
Similarly, let $\operatorname{ExVC}(n,k)$ be the number of families $\mathcal{F}\subseteq 2^{[n]}$ of VC dimension $k$ with $|tr(\mathcal{F})|=|\mathcal{F}|$. 
Theorem~\ref{Thm:SS} shows that for a family $\mathcal{F}\subseteq 2^{[n]}$ of VC dimension $k$, $|\mathcal{F}|=\sum_{i=0}^{k}\binom{n}{i}$ implies $|tr(\mathcal{F})|=|\mathcal{F}|$. 
Thus we have $f(n,k)\leqslant \operatorname{ExVC}(n,k)$. 
In fact, Balogh, M\'{e}sz\'{a}ros and Wagner~\cite{BMW2018} proved a further result. An induced matching is a matching such that no endpoints of two edges of the matching are joined by an edge of the graph. 
Let $\operatorname{IndMat}(n,k)$ denote the number of induced matchings in $Q_{n}$ between the layers $\binom{[n]}{k}$ and $\binom{[n]}{k+1}$, where $Q_{n}$ denotes the hypercube of dimension $n$. 
Balogh, M\'{e}sz\'{a}ros and Wagner~\cite{BMW2018} showed that the asymptotics of the logarithm of these two quantities and of $f(n,k)$ have the same value, as follows.
\begin{theorem}[Balogh, M\'{e}sz\'{a}ros and Wagner~\cite{BMW2018}]
Let $k=n^{o(1)}$. Then we have
\[
n^{(1+o(1))\binom{n}{k}}\leqslant \operatorname{IndMat}(n,k)\leqslant f(n,k)\leqslant \operatorname{ExVC}(n,k)\leqslant n^{(1+o(1))\binom{n}{k}}.
\]
\end{theorem}

\section{Forbidden Configurations in Matrices}
\label{sec:forbidden_config}
The main topic of this section is to reformulate the Sauer-Shelah Lemma through the language of forbidden configurations in {\it simple matrices}. Before proceeding, we require some basic definitions.

A {\it simple matrix} is a $(0,1)$-matrix with no repeated columns. Such matrices naturally correspond to set families in set theory. For a family $\mathcal{F} \subseteq 2^{[n]}$, it corresponds to a unique $n\times |\mathcal{F}|$ simple matrix $\mathbf{A} = (a_{ij})$, where rows index elements of $[n]$, columns index sets in $\mathcal{F}$, and $a_{ij} = 1$ if and only if element $i\in[n]$ belongs to the $j$-th set of $\mathcal{F}$. This matrix $\mathbf{A}$ is called the \textit{incidence matrix} of $\mathcal{F}$. 
Conversely, a simple matrix also corresponds to a unique set family.

\begin{definition}
    Given a fixed $k$-row $(0,1)$-matrix $\mathbf{F}$, we say a $(0,1)$-matrix $\mathbf{A}$ \textit{contains configuration $\mathbf{F}$} if there is a submatrix of $\mathbf{A}$ that is a row and column permutation of $\mathbf{F}$. We refer to $\mathbf{F}$ as a \textit{configuration} of $\mathbf{A}$ and write $\mathbf{F}\prec \mathbf{A}$.
\end{definition}

It can be checked that $\prec$ is a partial order on the set of all $(0,1)$-matrices. The following proposition shows that configurations align precisely with traces in set theory:
\begin{proposition}
    Let $\mathbf{F}$ and $\mathbf{G}$ be the incidence matrices of families $\mathcal{F}\subseteq 2^{[k]}$ and $\mathcal{G} \subseteq 2^{[n]}$ respectively, where $k\leqslant n$ are two positive integers. Then $\mathbf{G}$ avoids $\mathbf{F}$ as a configuration if and only if $\mathcal{G}$ forbids any isomorphic copy of $\mathcal{F}$ as a trace on some $k$-element subset of $[n]$.
\end{proposition}

\begin{definition}
    For a $(0,1)$-matrix $\mathbf{F}$, we denote $\operatorname{forb}(m, \mathbf{F})$ as the largest value such that there exists a simple $m \times \operatorname{forb}(m, \mathbf{F})$ matrix $\mathbf{A}$ such that $\mathbf{A}$ has no configuration $\mathbf{F}$.
\end{definition}

Note that $\mathbf{F}$ is not required to be simple when considering $\operatorname{forb}(m, \mathbf{F})$.
The following are two straightforward yet useful propositions for $\operatorname{forb}(m, \mathbf{F})$.
\begin{proposition}
Assume that $\mathbf{F}$ is a $(0,1)$-matrix.
    \begin{itemize}
        \item[(1)] Let $\mathbf{F}^{c}$ denote the $(0,1)$-complement of $\mathbf{F}$. Then we have $\operatorname{forb}(m, \mathbf{F})=\operatorname{forb}(m, \mathbf{F}^{c})$.
        \item[(2)] If $\mathbf{F}$ has $\mathbf{F}'$ as a configuration, then we have $\operatorname{forb}(m, \mathbf{F}')\leqslant \operatorname{forb}(m, \mathbf{F})$.
    \end{itemize}
\end{proposition}

Let $\mathbf{K}_k$ denote the complete configuration which contains all $2^k$ binary columns on $k$ rows.
Then Sauer-Shelah Lemma can be restated as follows:
\begin{theorem}[Sauer-Shelah Lemma]\label{conf K_k}
For the complete configuration $\mathbf{K}_k$, we have
\begin{equation}
\operatorname{forb}(m, \mathbf{K}_k) = \sum_{i=0}^{k-1} \binom{m}{i} = \Theta(m^{k-1})
\end{equation}
\end{theorem}

Let $\mathbf{K}_k^s$ denote the $k \times \binom{k}{s}$ simple matrix consisting of all possible columns with a column sum of $s$. Since $\mathbf{K}_{k}$ has $\mathbf{K}_k^s$ as a configuration, it is easy to see that $\operatorname{forb}(m, \mathbf{K}_k^s)\leqslant \operatorname{forb}(m, \mathbf{K}_k)$. F\"{u}redi and Quinn~\cite{FQ1983} proved that these two functions are eventually equal.

\begin{theorem}[F\"{u}redi and Quinn~\cite{FQ1983}]\label{conf K_k^s}
Let $k$ and $s$ be positive integers with $0 \leqslant s \leqslant k$. Then
\[
\operatorname{forb}(m, \mathbf{K}_k^s) = \binom{m}{k-1} + \binom{m}{k-2} + \cdots + \binom{m}{0}.
\]
\end{theorem}

Let $[\mathbf{A} \mid \mathbf{B}]$ denote the matrix obtained from concatenating the two matrices $\mathbf{A}$ and $\mathbf{B}$. For any simple matrix $\mathbf{A}$ and a positive integer $k$, let $k \cdot \mathbf{A}$ to denote the matrix $[\mathbf{A}|\mathbf{A}| \cdots \mid \mathbf{A}]$ consisting of $k$ copies of $\mathbf{A}$ concatenated together. The following theorem of Gronau~\cite{G1980} gives the exact value of $\operatorname{forb}(m, 2 \cdot \mathbf{K}_k)$, which is equal to $\operatorname{forb}(m, \mathbf{K}_{k+1})$.

\begin{theorem}[Gronau~\cite{G1980}]
We have
\[
\operatorname{forb}(m, 2 \cdot \mathbf{K}_k) = \binom{m}{k} + \binom{m}{k-1} + \cdots + \binom{m}{0}.
\]
\end{theorem}

The following theorem by Anstee and F\"{u}redi~\cite{AF1986} determines an upper bound of $\operatorname{forb}(m, t \cdot \mathbf{K}_k)$ for general $t$.
\begin{theorem}[Anstee and F\"{u}redi~\cite{AF1986}]\label{conf t.K_k}
We have
\[
\operatorname{forb}(m, t \cdot \mathbf{K}_k) = \operatorname{forb}(m, t \cdot \mathbf{K}_k^k) \leqslant \frac{t-2}{k+1}\binom{m}{k}(1-o(1)) + \binom{m}{k} + \binom{m}{k-1} + \cdots + \binom{m}{0}.
\]
\end{theorem}
For $t=1$, Theorem \ref{conf K_k} shows that $\operatorname{forb}(m, \mathbf{K}_k) = \Theta(m^{k-1})$. For $t\geqslant 2$, since $\operatorname{forb}(m, t \cdot \mathbf{K}_k)\geqslant \operatorname{forb}(m, 2 \cdot \mathbf{K}_k)=\operatorname{forb}(m, \mathbf{K}_{k+1})$, Theorem \ref{conf t.K_k} gives that $\operatorname{forb}(m, t \cdot \mathbf{K}_k)= \Theta(m^{k})$.
For further results on the exact bounds of $\operatorname{forb}(m, \mathbf{F})$ for specific configurations $\mathbf{F}$, 
we refer interested readers to the survey by Anstee~\cite{A2013}.

Recall that Theorem \ref{conf K_k^s} shows that $\operatorname{forb}(m, \mathbf{K}_k)=\operatorname{forb}(m, \mathbf{K}_k^s)$. This inspires us to think of what is the minimal $\mathbf{F}'$ under the partial order $\prec$ such that $\operatorname{forb}(m, \mathbf{F}')=\operatorname{forb}(m, \mathbf{F})$. 
To be more precise, Anstee and Karp~\cite{AK2010} introduced the following definition.
\begin{definition}
    A \emph{critical substructure} of a configuration $\mathbf{F}$ as a minimal configuration $\mathbf{F}'$ under the partial order $\prec$ such that $\mathbf{F}$ contains configuration $\mathbf{F}'$ and $\operatorname{forb}(m, \mathbf{F}')=\operatorname{forb}(m, \mathbf{F})$.
\end{definition}
Raggi~\cite{R2011} verified that the only $k$-rowed critical substructures of $\mathbf{K}_{k}$ are $\mathbf{K}_{k}^{s}$ for $s=0,1,\dots k$. Besides, Raggi~\cite{R2011} also determined the complete list of critical substructures for $\mathbf{K}_{4}$.
\begin{theorem}[Raggi~\cite{R2011}]
    The critical substructures of $\mathbf{K}_{4}$ are $\mathbf{0}_{4}$, $\mathbf{I}_{4}$, $\mathbf{K}_{4}^{2}$, $\mathbf{I}_{4}^{c}$, $\mathbf{1}_{4}$, $2\cdot \mathbf{0}_{3}$ and $2\cdot \mathbf{1}_{3}$.
\end{theorem}
We mention that the critical substructures for $\mathbf{K}_{5}$ have not been fully determined.

On the other hand, it is also worth considering the opposite scenario, namely, configurations $\mathbf{F}'$ that contain $\mathbf{F}$ and satisfy $\operatorname{forb}(m, \mathbf{F}') = \operatorname{forb}(m, \mathbf{F})$.
Let $\mathbf{1}_{p}\mathbf{0}_{q}$ denote the column consisting of $p$ $1$s at the top and $q$ $0$s at the bottom.
Anstee and Meehan~\cite{AM2011} showed that a column $\mathbf{1}_{p}\mathbf{0}_{q}$ can be added to $\mathbf{K}_k$ without changing the bound.
\begin{theorem}[Anstee and Meehan~\cite{AM2011}]
    Assume that $p+q=k$ and $p,q\geqslant 2$. Then
    \[
    \operatorname{forb}(m, [\mathbf{K}_k | \mathbf{1}_{p}\mathbf{0}_{q}])=\operatorname{forb}(m, \mathbf{K}_k).
    \]
\end{theorem}
Let $\mathbf{A}$ be an $m_{1}\times n_{1}$ matrix and $\mathbf{B}$ be an $m_{2}\times n_{2}$ matrix. Let $\mathbf{A}\times \mathbf{B}$ denote the $(m_{1}+m_{2})\times (n_{1}n_{2})$ matrix with all possible columns formed from one column from $\mathbf{A}$ placed on top of one column from $\mathbf{B}$. We let $\mathbf{K}_{2}^{T}$ be the following matrix:
\[
\mathbf{K}_{2}^{T}=\begin{pmatrix}
 1 &1 \\
 1 &0 \\
 0 &1 \\
 0 &0
\end{pmatrix}.
\]
Anstee and Nikov~\cite{AN2022} also proved the following theorem.
\begin{theorem}[Anstee and Nikov~\cite{AN2022}]
    Assume $k\geqslant 4$ and $t\geqslant 1$. There exists an $m_{k,t}$ such that for $m\geqslant m_{k,t}$ we have
    \[
    \operatorname{forb}(m, [\mathbf{K}_k | t\cdot (\mathbf{K}_{2}^{T}\times \mathbf{K}_{k-4})])=\operatorname{forb}(m, \mathbf{K}_k).
    \]
\end{theorem}

\section{Tur\'{a}n Numbers for Traces}\label{sec 8}

Tur\'{a}n's theorem~\cite{T1941}, a cornerstone of extremal graph theory, determines the Tur\'{a}n numbers for all cliques. 
Building on this, Tur\'{a}n-type problems seek to determine $\operatorname{ex}(n, F)$ for arbitrary graphs $F$, a framework that naturally extends to hypergraphs. 
Specifically, for a family $\mathfrak{G}$ of $r$-uniform hypergraphs, $\operatorname{ex}_r(n, \mathfrak{G})$ represents the maximum number of hyperedges in an $n$-vertex $r$-uniform hypergraph that avoids all members of $\mathfrak{G}$. 
However, for general $r$, this problem remains open and is widely regarded as highly challenging. 
Therefore, researchers have focused on more tractable variants of the problem. 
A notable example is the \emph{Berge copy}, a concept due to Berge~\cite{B1987}.
The following formal definition can be found in Gerbner and Palmer~\cite{GP2017}.

\begin{definition}
Let $\mathcal{F}$ and $\mathcal{T}$ be uniform hypergraphs (possibly with different uniformities) with $V(\mathcal{F})\subseteq V(\mathcal{T})$. We say that $\mathcal{T}$ is a \emph{Berge copy} of $\mathcal{F}$ if there exists a bijection $\phi:E(\mathcal{F})\to E(\mathcal{T})$ such that for every edge $e\in E(\mathcal{F})$, we have $e\subseteq \phi(e)$. Moreover, we say that $\mathcal{T}$ is an \emph{induced Berge copy} of $\mathcal{F}$ (or a trace of $\mathcal{F}$) if there exists a bijection $\phi:E(\mathcal{F})\to E(\mathcal{T})$ such that for every edge $e\in E(\mathcal{F})$, we have $\phi(e)\cap V(\mathcal{F})=e$.
\end{definition}
The following establishes a connection between induced Berge copies and traces of families.
\begin{proposition}
    If $\mathcal{T}$ is an induced Berge copy of $\mathcal{F}$, then we have $\mathcal{T}_{\mid V(\mathcal{F})}=\mathcal{F}$.
\end{proposition}
To proceed, we require the following definition provided by Luo and Spiro~\cite{LS2021}.
\begin{definition}
    Given a hypergraph $\mathcal{F}$, let $\operatorname{B}_{r}(\mathcal{F})$ denote the set of all $r$-uniform hypergraphs that are Berge copies of $\mathcal{F}$ up to isomorphism. Similarly, let $\operatorname{Tr}_{r}(\mathcal{F})$ denote the set of all $r$-uniform induced Berge copies of $\mathcal{F}$ up to isomorphism.
\end{definition}
It is evident that $\operatorname{Tr}_{r}(\mathcal{F})\subseteq \operatorname{B}_{r}(\mathcal{F})$. 
Taking an example, if $\mathcal{F}=K_{3}$ and $V(\mathcal{F})=[3]$, then for $r=3$ we have
\begin{align*}
&\operatorname{Tr}_{3}(K_{3})=\{\{124,134,234\}, \{124,134,235\}, \{124,135,236\}\}, \mbox{ and } \\
&\operatorname{B}_{3}(K_{3})=\{\{124,134,234\}, \{124,134,235\}, \{124,135,236\}, \{123,134,235\}\}.
\end{align*}

We also define the Tur\'{a}n number for these graph-based hypergraphs.
\begin{definition}
    Assume that $F$ is a graph, let $\operatorname{ex}(n,\operatorname{B}_{r}(F))$ (resp. $\operatorname{ex}(n,\operatorname{Tr}_{r}(F))$) be the maximum number of edges of an $n$-vertex $r$-uniform hypergraph that does not contain a subhypergraph isomorphic to a Berge copy (resp. an induced Berge copy) of $F$.
\end{definition}

The Tur\'{a}n problem for graph-based hypergraphs is closely related to the so-called \emph{generalized Tur\'{a}n problems}, which were first pioneered by Erd\H{o}s~\cite{E1962}. 
The generalized Tur\'{a}n problems were systematically studied by Alon and Shikhelman~\cite{AS2016}. 
\begin{definition}
    Let $F$ be a graph. Given two graphs $H$ and $G$, let $\mathcal{N}(H,G)$ be the number of copies of $H$ contained in $G$. We let $\operatorname{ex}(n,H,F)=\max\{\mathcal{N}(H,G):G \mbox{ is an } n\mbox{-vertex } F\mbox{-free graph}\}$.
\end{definition}
For a graph $G$, we define $\mathcal{N}_{r}(G)$ to be the $r$-uniform hypergraph that consists of all $r$-element vertex sets which span a copy of $K_{r}$ in $G$. Obviously, we have $\operatorname{ex}(n,K_{r},F)=\max\{|\mathcal{N}_{r}(G)|:G \mbox{ is }F\mbox{-free}\}$. If $G$ is $F$-free, then $\mathcal{N}_{r}(G)$ is $\operatorname{B}_{r}(F)$-free, implying $|\mathcal{N}_{r}(G)|\leqslant \operatorname{ex}(n,\operatorname{B}_{r}(F))$. Therefore, we can get the following relationship between these three functions:
\[
\operatorname{ex}(n,K_{r},F)\leqslant \operatorname{ex}(n,\operatorname{B}_{r}(F))\leqslant \operatorname{ex}(n,\operatorname{Tr}_{r}(F)).
\]
Sali and Spiro~\cite{SS2017} determined the order of magnitude of $\operatorname{ex}(n,\operatorname{Tr}_{r}(K_{s,t}))$ when $t\geqslant (s-1)!+1$ and $s\leqslant 2r-4$.
Later, F\"{u}redi and Luo~\cite{FL2023} determined the order of magnitude of $\operatorname{ex}(n,\operatorname{Tr}_{r}(F))$ for all graphs $F$ in terms of the generalized Tur\'{a}n numbers.
\begin{theorem}[F\"{u}redi and Luo~\cite{FL2023}]
Suppose that $r\geqslant 2$, $F$ is a graph with $E(F)\ne \emptyset$. Then, as $n\to \infty$, we have
\[
\operatorname{ex}(n,\operatorname{Tr}_{r}(F))=\Theta (\max_{2\leqslant s\leqslant r}\{\operatorname{ex}(n,K_{s},F)\}).
\]
\end{theorem}

The first to study forbidden induced Berge copies of graphs were Mubayi and Zhao~\cite{MZ2007}. They determined the asymptotic value of $\operatorname{ex}(n, \operatorname{Tr}_r(K_s))$ for $s \in \{3, 4\}$ or $r \in \{s, s+1\}$, as follows. 

\begin{theorem}[Mubayi and Zhao~\cite{MZ2007}]
Suppose that $s\in \{3,4\}$ or $r\in \{s,s+1\}$. Then
\[
\operatorname{ex}(n,\operatorname{Tr}_{r}(K_{s}))=\left(\frac{n}{s-1}\right)^{s-1}+o(n^{s-1}).
\]
\end{theorem}
Mubayi and Zhao~\cite{MZ2007} also conjectured that $\operatorname{ex}(n,\operatorname{Tr}_{r}(K_{s}))\sim (\frac{n}{s-1})^{s-1}$ holds for $s\geqslant 5$. 
In fact, they conjectured an extremal construction for $\operatorname{ex}(n,\operatorname{Tr}_{r}(K_{s}))$ for sufficiently large $n$ when $r\geqslant s$ and they proved this for the case $r=3$. Gerbner~\cite{G2023} confirmed this conjecture for the case $s=3$.
He~\cite{G2023} also extended the asymptotic bound to book graphs $K_{1,1,t}$ in the $3$-uniform case, which gave that $\operatorname{ex}(n,\operatorname{Tr}_{r}(K_{1,1,t}))=(1+o(1))n^{2}/4$.
The research of $\operatorname{ex}(n,\operatorname{Tr}_{r}(F))$ for specific graphs $F$, such as stars and $K_{2,t}$, can be found in \cite{FL2023,LS2021,QG2022,G2023}.

\section{Open Problems}\label{sec 9}

\subsection{Generalizing Theorem \ref{(4,13)}}
Recall that $\M(n,a,b)=\min \{m:(n,m)\rightarrow (a,b)\}$. The following construction given by Frankl and Wang~\cite{FW2024} shows that $\M(n,\ell+1,2^{\ell+1}-2^{\ell-1}+1)>\prod_{0\leqslant i<\ell}\left \lfloor \frac{n+\ell+i}{\ell} \right \rfloor$ for any positive integer $\ell$.
\begin{construction}
		Let $\ell$ be a positive integer. Let $X_{0},\dots ,X_{\ell-1}$ form a partition of $[n]$ with $|X_{i}|=\left \lfloor \frac{n+i}{\ell} \right \rfloor$ for $0\leqslant i<\ell$. Define
\[
\mathcal{F}(n,\ell)=\{F\subseteq [n]: |F\cap X_{i}|\leqslant 1, 0\leqslant i<\ell\}.
\]
\end{construction}
It is easy to see that $|\mathcal{F}(n,\ell)|=\prod_{0\leqslant i<\ell}\left \lfloor \frac{n+\ell+i}{\ell} \right \rfloor$. For any $T\in \binom{[n]}{\ell+1}$, by Pigeonhole Principle, there exist two vertices $x,y\in [n]$ such that $x,y\in X_{i_{0}}$ for some $0\leqslant i_{0}<\ell$. Since $\mathcal{F}(n,\ell)_{\mid T}\subseteq 2^{T}\setminus \{F\subseteq T:F\supseteq \{x,y\}\}$, we can conclude that $|\mathcal{F}(n,\ell)_{\mid T}|\leqslant 2^{\ell+1}-2^{\ell-1}$. Therefore, for any positive integer $\ell$, this construction shows that
\[
\left(n,\prod_{0\leqslant i<\ell}\left \lfloor \frac{n+\ell+i}{\ell} \right \rfloor\right)\not\rightarrow (\ell+1,2^{\ell+1}-2^{\ell-1}+1),
\]
\noindent providing the desired lower bound for $\M(n,\ell+1,2^{\ell+1}-2^{\ell-1}+1)$. 
Frankl and Wang~\cite{FW2024} made the following conjecture, which asks whether the above lower bound is best possible.
\begin{conjecture}[Frankl and Wang~\cite{FW2024}]\label{l+1}
		For any positive integers $n>\ell>0$, we have
\[
\left(n,1+\prod_{0\leqslant i<\ell}\left \lfloor \frac{n+\ell+i}{\ell} \right \rfloor\right)\rightarrow (\ell+1,2^{\ell+1}-2^{\ell-1}+1).
\]
\end{conjecture}
Note that Conjecture~\ref{l+1} holds true for $\ell \leqslant 3$, as follows: For $\ell=1$, $(n,n+2)\rightarrow (2,4)$ is just from Sauer-Shelah Lemma; for $\ell=2$, $\left (n,\left \lfloor \frac{n^{2}}{4} \right \rfloor+n+2 \right )\rightarrow (3,7)$ is proved by Frankl~\cite{F1983}; finally, for $\ell=3$, Frankl and Wang~\cite{FW2024} show that $(n,1+\left \lfloor \frac{n+5}{3} \right \rfloor\left \lfloor \frac{n+4}{3} \right \rfloor\left \lfloor \frac{n+3}{3} \right \rfloor)\rightarrow (4,13)$ for $n\geqslant 25$.

\subsection{Shattering in antichains}
Recall that a family $\mathcal{F}$ is an \emph{antichain} if for any $F,F'\in \mathcal{F}$, $F\nsubseteq F'$ holds. 
The following longstanding conjecture was proposed by Frankl~\cite{F1989} thirty years ago.
\begin{conjecture}[Frankl~\cite{F1989}]\label{antichain}
Let $k$ be a non-negative integer and $n\geqslant 2k$. Suppose that $\mathcal{F}\subseteq 2^{[n]}$ is an antichain with $\mathcal{F} \not\rightarrow (k+1,2^{k+1})$. Then $|\mathcal{F}|\leqslant \binom{n}{k}$.
\end{conjecture}
Conjecture~\ref{antichain} was proved by Frankl~\cite{F1989} for $k\leqslant 2$ and by Anstee and Sali~\cite{AS1997} for $k=3$.

\subsection{The structure of extremal families}
A family $\mathcal{F}\subseteq 2^{[n]}$ is \emph{shattering-extremal}, or \emph{s-extremal} for short, if $|tr(\mathcal{F})|=|\mathcal{F}|$.
The main problem concerning $s$-extremal families is to find good characterizations of them.
A potential approach to this problem was proposed by
M\'{e}sz\'{a}ros and R\'{o}nyai~\cite{MR2012} in the following conjecture. 
\begin{conjecture}[M\'{e}sz\'{a}ros and R\'{o}nyai~\cite{MR2012}]\label{s-extremal}
		For every nonempty s-extremal family $\mathcal{F}\subseteq 2^{[n]}$, there exists $F_{0}\in \mathcal{F}$ such that $\mathcal{F}\setminus \{F_{0}\}$ is still s-extremal.
\end{conjecture}
By Theorem~\ref{tr equ str}, we know that $\mathcal{F}$ is $s$-extremal if and only if $2^{[n]} \setminus \mathcal{F}$ is also $s$-extremal.
Therefore, it is equivalent to consider whether there exists $F_{0}\notin \mathcal{F}$ such that $\mathcal{F}\cup \{F_{0}\}$ is s-extremal. 
It is obvious that hereditary families trivially satisfy Conjecture \ref{s-extremal}. 
In \cite{MR2012, MR2014}, M\'{e}sz\'{a}ros and R\'{o}nyai confirmed this conjecture for all families with VC dimension at most 2.

\subsection{Forbidden traces of complete graphs}

Let $K_{s}^{(t)}$ be the complete $s$-vertex $t$-uniform hypergraph. Recall that for a hypergraph $F$, we let $\operatorname{Tr}_{r}(F)$ denote the set of all $r$-uniform induced Berge copies of $F$ up to isomorphism. Consider the family $\operatorname{Tr}_{r}(K_{s}^{(t)})$, let $H_{t,s}^{r}$ be the member of $\operatorname{Tr}_{r}(K_{s}^{(t)})$ with the maximum number of vertices. In other words, $H_{t,s}^{r}$ is the $r$-uniform hypergraph obtained from $K_{s}^{(t)}$ by adding $r-t$ new vertices to each hyperedge such that each new vertex is added to only one hyperedge. In particular, $H_{r,s}^{r}=K_{s}^{(r)}$.

By definition, we have $\operatorname{ex}(n, \operatorname{Tr}_{r}(K_{s}^{(t)}))\leqslant \operatorname{ex}(n, H_{t,s}^{r})$. In order to determine the value of $\operatorname{ex}(n, \operatorname{Tr}_{r}(K_{s}^{(t)}))$, Mubayi and Zhao~\cite{MZ2007} presented the following conjecture.
\begin{conjecture}[Mubayi and Zhao~\cite{MZ2007}]\label{turan trace conj}
Given positive integers $n$, $r$, $s$ and $t$ with $2\leqslant t<\min \{r,s\}$. There exists $n_{0}>0$ such that for $n>n_{0}$, we have
\[
\operatorname{ex}(n, \operatorname{Tr}_{r}(K_{s}^{(t)}))=
\begin{cases}
\operatorname{ex}(n, H_{t,s}^{r}) &\mbox{ if }r<s, \\
\operatorname{ex}(n-r+s-1, \operatorname{Tr}_{s-1}(K_{s}^{(t)}))=\operatorname{ex}(n-r+s-1, H_{t,s}^{s-1}) &\mbox{ if }r\geqslant s.
\end{cases}
\]
\end{conjecture}
Mubayi and Zhao~\cite{MZ2007} observed that for $r<s$ and sufficiently large $n$, $T_{s-1}^{r}$ is an extremal family for both $\operatorname{ex}(n, \operatorname{Tr}_{r}(K_{s}))$ and $\operatorname{ex}(n, H_{2,s}^{r})$, where $T_{\ell}^{r}$ is the complete $\ell$-partite $r$-uniform hypergraph. 
In this case, it holds that $\operatorname{ex}(n, \operatorname{Tr}_{r}(K_{s}))=\operatorname{ex}(n, H_{2,s}^{r})$. 
So Conjecture \ref{turan trace conj} holds for the cases when $t=2$ and $r<s$.
As for the cases when $r\geqslant s$, Mubayi and Zhao~\cite{MZ2007} showed that Conjecture \ref{turan trace conj} holds for $t=2$ and $r=s=3$, that is, $\operatorname{ex}(n, \operatorname{Tr}_{3}(K_{3}))=\operatorname{ex}(n-1, K_{3})=\left \lfloor \frac{(n-1)^{2}}{4} \right \rfloor.$
Besides, they also showed that $\operatorname{ex}(n, \operatorname{Tr}_{r}(K_{3}))=n^{2}/2+o(n^{2})$, which confirmed Conjecture \ref{turan trace conj} asymptotically if $s=3$ and $t=2$. Later, Gerbner~\cite{G2023} solved the case when $t=2$ and $s=3$, by showing for sufficiently large $n$, $\operatorname{ex}(n, \operatorname{Tr}_{r}(K_{3}))=\left \lfloor \frac{(n-r+2)^{2}}{4} \right \rfloor.$

\bibliographystyle{spmpsci} 
\bibliography{references} 

\end{document}